\documentclass[12pt]{amsart}
\usepackage{amsmath,amssymb}

\title{Cycle class maps and birational invariants}
\author{Brendan Hassett and Yuri Tschinkel}
\date{August 1, 2019}

\newcommand{\bC}{\mathbb C}

\newcommand{\bN}{\mathbb N}
\newcommand{\bP}{\mathbb P}
\newcommand{\bQ}{\mathbb Q}
\newcommand{\bR}{\mathbb R}
\newcommand{\bZ}{\mathbb Z}

\newcommand{\cC}{\mathcal C}
\newcommand{\cD}{\mathcal D}
\newcommand{\cO}{\mathcal O}

\newcommand{\cZ}{\mathcal Z}

\newcommand{\ra}{\rightarrow}

\newcommand{\Alb}{\operatorname{Alb}}

\newcommand{\CH}{\operatorname{CH}}
\newcommand{\Chow}{\operatorname{Chow}}
\newcommand{\Ext}{\operatorname{Ext}}
\newcommand{\Gal}{\operatorname{Gal}}
\newcommand{\Gr}{\operatorname{Gr}}
\newcommand{\Griff}{\operatorname{Griff}}
\newcommand{\Hg}{\operatorname{Hg}}

\newcommand{\Mor}{\operatorname{Mor}}
\newcommand{\NS}{\operatorname{NS}}
\newcommand{\Pic}{\operatorname{Pic}}
\newcommand{\Spec}{\operatorname{Spec}}
\newcommand{\Sym}{\operatorname{Sym}}

\theoremstyle{plain}

\theoremstyle{plain}
\newtheorem{prop}{Proposition}

\newtheorem{theo}[prop]{Theorem}
\newtheorem{coro}[prop]{Corollary}

\theoremstyle{definition}
\newtheorem{defi}[prop]{Definition}

\newtheorem{ques}[prop]{Question}

\newtheorem{rema}[prop]{Remark}

\newtheorem{exam}[prop]{Example}

\begin{document}

\begin{abstract}
We introduce new obstructions to rationality for geometrically
rational threefolds arising from the geometry of curves and
their cycle maps. 
\end{abstract}

\maketitle

\section{Introduction}
Let $X$ be a smooth projective variety over a field $k\subset \bC$
with $X_{\bC}$ rational. When is $X$ rational over $k$? 

It is necessary that $X(k)\neq \emptyset$. This is also sufficient if $X$ has dimension $1$. 
In dimension $2$, this is not sufficient, but there are
effective criteria for rationality, due to Enriques, Iskovskikh, Manin, and others.
For example, minimal del Pezzo surfaces of degree $\le 4$ are never 
rational. Indeed, the Galois action on the N\'eron-Severi group -- the
lines especially -- governs the rationality of $X$.

The case of threefolds remains open. The Galois action on the N\'eron-Severi
group can still be used to obtain nonrationality in some cases, but
never when that group has rank one or is split over the ground field.
The case of complete intersections of two quadrics was considered in
depth in \cite{HT2quad}; we gave a complete characterization of
rationality over $k=\bR$. Benoist and Wittenberg \cite{BenWit} developed
an approach inspired by the Clemens-Griffiths method of intermediate
Jacobians. If a threefold $X$ is rational then its cohomology reflects
invariants of curves blown up in parametrizations 
$\bP^3 \dashrightarrow X$. Over $\bC$, the intermediate Jacobian of $X$
must be isomorphic to a product of Jacobians of curves. When $k$ is
not algebraically closed, one may endow the
intermediate Jacobian with the structure of a principally polarized
abelian variety over $k$ \cite{ACMV2}.  This must be isomorphic 
to a product of Jacobians of (not necessarily 
geometrically connected) curves over $k$, if $X$ is
to be rational over $k$. Benoist and Wittenberg exhibit geometrically
rational conic bundles over $\bP^2$ where the latter condition fails to hold,
e.g., over $k=\bR$. 

Now suppose that $X$ is a smooth projective geometrically rational
threefold as above, with rank-one N\'eron-Severi group and 
intermediate Jacobian $J^2(X)$
isomorphic to a product of Jacobians of curves
over $k$. We introduce new obstructions to rationality
of such $X$ over $k$ based on the geometry of curves on $X$ and provide
examples where they apply. 

The idea is that the cycle class
map on curves of given degree $d$ naturally takes its values
in a principal homogeneous space over $J^2(X)$. Moreover,
if $J^2(X)\simeq J^1(C)$ for a smooth geometrically irreducible
curve $C$ of genus $g\ge 2$ over $k$, then this homogeneous space is equivalent
to a component of the Picard scheme of $C$ provided $X$ is rational over
$k$. This is a strong constraint as the order of any such component
divides $2g-2$.

As an application, we completely characterize (in Theorem~\ref{theo:twoquad})
the rationality of smooth intersections
of two quadrics $X\subset \bP^5$: It is necessary and sufficient that $X$ admit a line
over the ground field.

\

Here is a roadmap of the paper: We review constructions of cycle class
maps over the complex numbers in Section~\ref{sect:complex}; this
serves as motivation for our arithmetic approach. Cycle class
maps take values in abelian varieties; we discuss Albanese
morphisms from singular varieties in Section~\ref{sect:Alb}.
We turn to nonclosed fields in Section~\ref{sect:passage},
discussing how to define cycle maps over the relevant fields
of definition. Our approach is a geometric implementation of
the $\ell$-adic Abel-Jacobi map studied by Jannsen.
The key invariant is presented in 
Section~\ref{sect:construct}, in arbitrary dimensions.
An application to threefolds can be found
in Section~\ref{sect:threefolds}.

\

It would be interesting to find nontrivial
examples of geometrically rational 
fourfolds where this machinery applies. Which principal homogeneous
spaces over abelian varieties are realized by zero-cycles on
curves and surfaces?

\

\noindent
{\bf Acknowledgments:}
The first author was partially supported by NSF grants 1551514 and 1701659, and the Simons Foundation;
the second author was partially supported by NSF grant 1601912.
We are grateful to Jeffrey Achter and Olivier Wittenberg
for conversations on this topic.

\section{Review of the complex case}
\label{sect:complex}
Let $X$ be a smooth complex projective variety.

\subsection{Cycle class maps}
Regard $X$ as a complex manifold.
We consider {\em Deligne cohomology}, following
\cite[\S 1]{EV}. For each
integer $p\ge 0$ we have the complex $\bZ(p)_{\cD}$
of complex analytic sheaves
$$0 \ra \bZ(p) \ra \cO_X \ra \Omega^1_X \ra \cdots \ra \Omega^{p-1}_X
\ra 0,$$
where $\bZ(p) \ra \cO_X$ takes $1$ to $(2\pi i)^p$ and the subsequent
arrows are exterior differentiation. Deligne cohomology is defined
as the hypercohomology of this complex
$$H^q_{\cD}(X,\bZ(p)):=\mathbb{H}^q(\bZ(p)_{\cD}).$$
For $p=0$ we recover ordinary singular cohomology
$$H^q_{\cD}(X,\bZ(0))=H^q(X,\bZ).$$
When $p=q=1$ we have
$$H^1_{\cD}(X,\bZ(0))=H^0(X,\cO^*_X).$$
The exponential exact sequence gives
$$H^2_{\cD}(X,\bZ(1))=H^1(X,\cO^*_X)=\operatorname{Pic}(X).$$

\

More generally, there is a cycle-class mapping \cite[\S 7]{EV}
$$\psi^p: \CH^p(X) \ra H^{2p}_{\cD}(X,\bZ(p)).$$
The target fits into a short exact sequence \cite[7.9]{EV}
$$0 \ra J^p(X) \ra H^{2p}_{\cD}(X,\bZ(p)) \ra \Hg^p(X)
\ra 0. $$
Here the right term is the {\em Hodge cycles}, the kernel of
the homomorphisms
$$H^{2p}(X,\bZ(p)) \ra H^{2p}(X,\bC) \ra
 \oplus_{j=0,\ldots,p-1} H^{2p-j}(X,\Omega^j_X)$$
coming from Hodge theory. The left term is the {\em intermediate
Jacobian}, a complex torus
$$J^p(X):=H^{2p-1}(X,\bC) / \left(H^{2p-1}(X,\bZ(p)) 
\oplus_{j=p}^{2p-1} H^{2p-1-j}(X,\Omega^j_X)\right).$$

The cycle class map admits an interpretation in terms of extensions of
mixed Hodge structures
\cite[\S 9.1]{JannsenBook} that is useful in drawing comparisons
among cohomology theories. Suppose that $Z\subset X$ is a codimension-$p$
compact complex submanifold with complement $U=X\setminus Z$. 
Consider the exact sequence for cohomology with supports
$$ \cdots \ra H^{2p-1}(X,\bZ) \ra H^{2p-1}(U,\bZ) \ra H^{2p}_{|Z|}(X,\bZ)
\ra H^{2p}(X,\bZ) \cdots
$$
and the associated mixed Hodge structure on $U$. This is an extension
of the pure weight $2p$ Tate Hodge structure associated with the cycle class of 
$Z$ by the degree-$(2p-1)$ cohomology of $X$, yielding an element
\begin{equation} \label{eq:eta}
 \eta(Z) \in \Ext^1_{MHS}(\bZ(-p),H^{2p-1}(X,\bZ))\simeq J^p(X),
\end{equation}
where the last identification is discussed in \cite{Carlson}.

\subsection{Algebraicity and cycle maps}
\label{subsect:Griffiths}
Let $B$ be a smooth complex variety and
$$\begin{array}{ccc}
\cZ &\hookrightarrow & X\times B \\
\downarrow  & & \\
B & & 
\end{array}
$$
a flat family of codimension-$p$ subschemes. Then the induced cycle map
$$\Psi^p_B: B \ra H^{2p}_{\cD}(X,\bZ(p))$$
has the following properties:
\begin{itemize}
\item{the image lies in a coset $I$ for 
$J^p(X) \subset H^{2p}_{\cD}(X,\bZ(p))$;}
\item{the induced map $B\ra I$ is holomorphic for
the complex structure associated with an identification
$P\simeq J^p(X)$;}
\item{the smallest complex torus $P_B \subset I$
containing the image of $B$ 
carries the structure of an abelian variety and the 
induced $B \ra P_B$ is algebraic with respect to that structure.}
\end{itemize}
The first statement is clear as $B$ is connected.
The second may be found in \cite[Ap.~A]{GriffithsIHES}.
For the third, take the closure $\overline{B}$ of $B$ in
the Hilbert scheme and choose a projective resolution of
singularities
$\beta:\tilde{B}\ra \overline{B}$ that leaves $B$ unchanged.
Thus we have a flat family of cycles
$$\widetilde{\cZ} \ra \tilde{B}$$
and an induced proper holomorphic
$\Psi^p_{\tilde{B}}$ extending $\Psi^p_{B}$. Note that
$P_{\tilde{B}}=P_B$ is dominated by the Albanese
$\Alb(\tilde{B})$, thus is an abelian variety. Since
$\Psi^p_{\tilde{B}}$ is a holomorphic map of projective
varieties it is algebraic, thus $\Psi^p_B$ is algebraic
as well.

We fix $X$ and $p$ as above and consider families of
codimension-$p$ cycles $Z\subset X$.
Each family
yields a translate of an abelian subvariety of  $J^p(X)$. 
Let 
$J^p_{cyc}(X) \subset J^p(X)$ 
denote the distinguished maximal (connected) abelian subvariety 
arising from such families of cycles. We have
$$J^p_{cyc}(X) \subset E^p(X):=\psi^p(\CH^p(X)) \subset H^{2p}_{\cD}(X,\bZ(p))$$
and the quotient $G^p(X)$ of the second group by the first is countable,
as there are countably-many irreducible components of the Hilbert
scheme parametrizing subschemes of $X$.

Recall the Griffiths
group \cite{GriffithsAnnals}
$$\CH^p(X)_{hom}/\CH^p(X)_{alg}=:\Griff^p(X) \subset B^p(X):=
\CH^p(X)/\CH^p(X)_{alg}.
$$
We have a surjective homomorphism
$$B^p(X) \twoheadrightarrow G^p(X)$$
with kernel consisting of cycles Abel-Jacobi equivalent to zero.
Thus we obtain a diagram
$$\begin{array}{ccccccccc}
0 & \ra & \CH^p(X)_{alg} & \ra & \CH^p(X)    & \ra & B^p(X) & \ra    & 0 \\
  &     &    \downarrow  &     & \downarrow  &     & \parallel &     &   \\
0 & \ra &   J^p_{cyc}(X) & \ra &  D^p(X)     & \ra & B^p(X) & \ra     & 0  \\
  &     &    \parallel   &     & \downarrow  &     &\downarrow  &  & \\
0 & \ra &   J^p_{cyc}(X) & \ra &  E^p(X)     & \ra & G^p(X)     &\ra & 0  
\end{array}
$$
where the second row is induced by the third row.
We summarize this as follows:
\begin{prop}
The cycle class map induces homomorphisms
$$
\psi^p:\CH^p(X) \stackrel{\chi^p}{\longrightarrow} D^p(X) \twoheadrightarrow   E^p(X),
$$
where $E^p(X)$ (resp.~$D^p(X)$) is an extension of a countable group $G^p(X)$
(resp.~$B^p(X)$) by an abelian variety $J^p_{cyc}(X)$. 

Given a family of 
cycles over a connected base $B$, there is an
algebraic morphism 
$$B \ra P$$ 
to a principal homogeneous space for
$J^p_{cyc}(X)$.
\end{prop}

\subsection{Vanishing results}
\label{subsect:vanishing}
Let $X$ have dimension $n$.
\begin{itemize}
\item{$\Griff^1(X)=0$ by the Lefschetz $(1,1)$ theorem;}
\item{$\Griff^n(X)=0$ as all zero-cycles of degree zero
are algebraically trivial.}
\end{itemize}
We recall further results along these lines.
\begin{defi}
We say that $X$ admits a {\em decomposition of the diagonal} if
there exists a point $x\in X$, an $N\in \bN$, and a 
rational equivalence on $X\times X$
$$N\Delta_X \equiv N \{x\} \times X + Z',$$
where $Z'$ is supported on $X \times D$ for some
subvariety $D \subsetneq X$.
\end{defi}
Rationally connected varieties admit decompositions
of the diagonal.
\begin{theo} \cite[Thm.~I(i)]{BlSr}
\label{theo:Chowtwo}
If $X$ admits a
decomposition of the diagonal
then $\psi^2$
is an isomorphism. Thus we obtain an
isomorphism of abelian groups
$$\psi^2_{\circ}: \CH^2(X)_{hom} \stackrel{\sim}{\ra} J^2(X).$$
\end{theo}
This actually holds under weaker assumptions: It suffices
that the Chow group of zero-cycles on $X$ be supported on
a curve. Furthermore, $\Griff^2(X)=0$ provided the Chow group
of zero-cycles on $X$ is supported on a surface \cite[Thm.~I(ii)]{BlSr}.

Let $X$ be rationally connected of dimension $n$. Voisin has asked
whether $\Griff^{n-1}(X)$ always vanishes when $X$ is Fano.
This is known for certain 
complete intersections:
\begin{theo} \cite[Th.~1.7]{TianZong} \label{theo:TZ}
Let $X\subset \bP^{n+c}$ be a complete intersection of hypersurfaces
of degrees $d_1,\ldots,d_c$ with $d_1+\cdots+d_c \le n-1$. Then every
rational curve on $X$ is algebraically equivalent to an effective
sum of lines.
\end{theo}
As varieties of lines on complete intersections
are connected when their
expected dimension is positive, this implies $\Griff^{n-1}(X)=0$.

\subsection{Chow varieties}
\label{subsect:Chow}
Let $\Chow^p(X)$ denote the monoid of effective codimension-$p$ cycles on $X$
and $\Chow^p_d(X)$ the cycles of degree $d$ for each $d\in \Hg^p(X)$.
This carries the structure of a projective semi-normal scheme 
\cite[Th.~3.21]{KollarBook}. There is a well-defined addition operation
$$\Chow^p(X) \times \Chow^p(X) \ra \Chow^p(X)$$
endowing $\Chow^p(X)$ with the structure of a monoid.

Thus we obtain surjective homomorphisms 
$$
C^p(X) \twoheadrightarrow B^p(X) \twoheadrightarrow  G^p(X)
\twoheadrightarrow \CH^p(X)/\CH^p(X)_{hom},$$
where the first three groups are countably-generated 
and the last is finitely-generated.
Furthermore, $\Griff^p(X)$ is a subquotient of $C^p(X)$ -- cycles
parametrized by a connected component of the Chow variety are 
algebraically equivalent to each other.

For each ample divisor $h$ on $X$, we have a filtration by 
finitely-generated subgroups
$$F_nC^p(X) = \bigoplus_{\cC \subset \Chow^p_d(X) \text{ such that } h\cdot d \le n }\bZ[\cC]$$
This gives $C^p(X), B^p(X)$, and $G^p(X)$ compatible
structures of inductive limits of finitely
generated groups, independent of the choice of $h$.
The holds for $\Griff^p(X)$, i.e., we restrict to cycles 
expressible as sums of terms of bounded degree.

\begin{prop} \label{prop:Chowmap}
There is a unique 
cycle class morphism of complex analytic spaces
$$
\Psi^p_d: \Chow^p_d(X) \ra I,
$$
where $I$ is a coset for $J^p(X) \subset H^{2p}_{\cD}(X,\bZ(p)).$
Its image generates a finite union of translates of abelian 
subvarieties in the intermediate Jacobian, 
each contained in $J^p_{cyc}(X)$.
\end{prop}
\begin{proof}
As the Chow variety is seminormal by definition and the cycle class
is well-defined set-theoretically, it suffices to construct $\Psi^p_d$
on each connected component $W$ of the normalization.

In Section~\ref{subsect:Griffiths} we discussed how to define the
desired projective morphism on a resolution $\beta:\tilde{W} \ra W$.
As it is constant on the fibers of $\beta$, Stein factorization gives 
the desired descent to $W$.
\end{proof}

\section{Albanese varieties}
Here we work over a field $k\subset \bC$. Our goal
is to recast classical work of Lang,
Serre \cite{SerreMorphism}, and others,
with a view toward analyzing cycle maps.

\label{sect:Alb}
\begin{prop} \label{prop:constructAlb}
Let $T$ be projective, geometrically reduced, and geometrically
connected, over $k$.
Then there exists an abelian variety $\Alb(T)$, a principal
homogeneous space $P$ over $\Alb(T)$, and a morphism
$$i_T: T \ra P,
$$
all defined over $k$, 
with the following properties:
\begin{itemize}
\item{given a morphism $T \ra T'$ over $k$, where $T$ and $T'$
satisfy our hypotheses, there is an induced morphism
$$\Alb(T) \ra \Alb(T');$$}
\item{each morphism $T\ra P'$ to a principal homogeneous space
over an abelian variety defined over $k$ admits a factorization
through $i_T$.}
\end{itemize}
\end{prop} 
The construction of the Albanese is proven
using the arguments of \cite[Ap.~A]{Mochizuki} and \cite[Exp.~10,\S~4]{SerreMorphism}.
\begin{proof}
Choose a projective
resolution of singularities $\beta:\widetilde{T}\ra T$ for
which $\Alb(\widetilde{T})$ can be constructed by standard techniques.
(The Albanese variety 
of a disjoint union is the product of the Albanese varieties
of its components.)  We obtain
$\Alb(T)$ by quotienting by the smallest (not-necessarily connected)
abelian subvariety containing all subschemes contracted
under $\beta$. Note that the formation of $\Alb(T)$ is compatible
with field extensions.

We explain why this is independent of the choice of resolution: 
Given a birational morphism
of resolutions $\widehat{T} \ra \widetilde{T}$ over $T$, there is an induced
isomorphism $\Alb(\widehat{T}) \stackrel{\sim}{\ra} \Alb(\widetilde{T})$.
The abelian subvarieties generated by subschemes contracted by the
resolutions are identified. Note that the Albanese can be
computed similarly using {\em any} dominant morphism $\widetilde{T} \ra T$
from a smooth projective variety.   

To prove functoriality, we present the Albanese varieties using
a diagram
$$\begin{array}{ccc}
\widetilde{T} & \ra & \widetilde{T'} \\
\downarrow &    & \downarrow \\
    T      &  \ra & T'
\end{array}$$
where the vertical arrows are dominant and the varieties in the 
upper row are smooth and projective. The natural
$$\Alb(\widetilde{T}) \ra \Alb(\widetilde{T'})$$
induces the desired homomorphism. 

Connectedness is used to obtain a well-defined morphism
from $T$ into a principal
homogeneous space over $\Alb(T)$. 
Choose a strict normal-crossings resolution $\widetilde{T} \ra T$ defined
over $k$.
Choose a finite extension $L/k$ over which each geometrically irreducible
stratum of $\widetilde{T}$
is defined and admits a rational point over $L$. Fixing a base
point on each irreducible component of $\widetilde{T}$,
yields a morphism
$$\widetilde{T}_L \ra \Alb(\widetilde{T})_L,$$
that corresponds to the standard map on each component.
(We take each base point to $0$.)
Since $T$ is geometrically connected, given $t \in T(L)$, we
can translate the cycle maps on each component so that they glue
together to a well-defined
$$\begin{array}{rcl}
T_L  & \ra & \Alb(T)_L  \\
   t  & \mapsto & 0.
\end{array}
$$
Fix the irreducible component 
$$P \subset
\Mor(T,\Alb(T))$$
containing these morphisms,
a principal homogeneous space over $\Alb(T)$
with respect to the induced translation action.
The morphism
$$i_T:T \ra P$$
takes $t\in T$ to the morphism mapping $t$ to the identity.

The universal property of $i_T$ follows:
Given $T \ra P'$ functoriality gives a
homomorphism $\Alb(T) \ra \Alb(P')$ such that
$[P] \mapsto [P']$, as cocycles for the corresponding 
Albanese varieties. We thus obtain
$$T \stackrel{i_T}{\ra} P \ra P',$$
the desired factorization.
\end{proof}

We recall an elementary property of principal homogeneous spaces:
Let $P$ and $P'$ be principal homogenous
spaces over an abelian variety $J$. Then multiplication induces
a morphism
$$P \times_{\Spec(k)} P' \ra P''$$
to a homogeneous space $P''$ over $J$ satisfying
$$[P]+[P']=[P''].$$

\begin{coro} \label{coro:symAlb}
Retain the notation of Proposition~\ref{prop:constructAlb}. 
For each $d \in \bN$ there is a natural morphism
$$i_T^d:\Sym^d(T) \ra P_d,$$
where $P_d$ is the principal homogeneous space over $\Alb(T)$, 
satisfying
$$[P_d]=d[P]$$
in the Weil-Ch\^atelet group of $\Alb(T)$. 
When $d\gg 0$ the morphism $i_T^d$ is dominant. 
\end{coro}
Hence for $d$ sufficiently
large and divisible -- e.g., when $T$ admits a point over a
degree $d$ extension -- $\Alb(T)$ is dominated by $\Sym^d(T)$.
\begin{proof}
Indeed, $i_T$ gives
$$\Sym^d(T) \ra \Sym^d(P)$$
and addition induces
$$\underbrace{P\times \cdots \times P}_{d \text{ times }} \ra P_d,$$
compatible with permutations of the factors. 

For the last statement: If $W$ is smooth, projective, and geometrically 
integral then $i_W^e:\Sym^e(W) \ra P_{e,W}$ -- the morphism onto
the degree-$e$ torsor over $\Alb(W)$ -- is dominant for large $e$. 
Let $d$ be the sum of the $e$'s taken over all (geometrically) irreducible
components of a resolution of $T$. As $\Alb(T)$ is a quotient of the
product of the Albanese varieties of these components, we find that
$i_T^d$ is dominant as well.
\end{proof}

\section{Passage to nonclosed fields}
\label{sect:passage}

We continue to work over a field 
$$k \subset \bar{k} \subset \bC$$
with absolute Galois group $\Gamma=\Gal(\bar{k}/k)$.
Let $X$ be a smooth projective variety over $k$
and write
$\bar{X}=X_{\bar{k}}$. 

\subsection{$\ell$-adic cycle maps}
One formulation goes back to Bloch \cite{Bloch}:
Let $\ell$ be a prime and
$\CH^p(\bar{X})(\ell)$ the $\ell$-primary part of the torsion, i.e.,
$$\CH^p(\bar{X})(\ell)=\varinjlim_{\nu \ra \infty} \CH^n(\bar{X})[\ell^{\nu}].$$
Then there is a functorial cycle class homomorphism
$$\lambda^p_{\ell}:\CH^p(\bar{X})(\ell) \ra H^{2p-1}(\bar{X},\bQ_{\ell}/\bZ_{\ell}(p)).$$
It is an isomorphism when $p=n:=\dim(X)$, in which case the target
may be interpreted as the $\ell$-primary part of the torsion of the
Albanese $\Alb(\bar{X})$ \cite[3.9]{Bloch}.

Jannsen \cite[\S 3]{JannsenMA}
has defined cycle class maps
to continuous \'etale cohomology 
$$\psi^p_{\ell}:\CH^p(X) \ra H^{2p}_{cont}(X,\bZ_{\ell}(p)).$$
The main advantage of continuous cohomology 
is the existence of a Hochschild-Serre
type spectral sequence under field extensions. 
Fix the cohomology class of an algebraic cycle
$$[Z_0] \in H^0_{cont}(\Gamma,H^{2p}(\bar{X},\bZ_{\ell}(p)))$$
and consider the induced 
\begin{equation} \label{eq:Jannsen}
\psi^p_{\ell}:\{Z \in \CH^p(X):[Z]=[Z_0] \} \ra
H^1_{cont}(\Gamma,H^{2p-1}(\bar{X},\bZ_{\ell}(p))),
\end{equation}
the $\ell$-adic analog of the Abel-Jacobi map \cite[\S 6]{JannsenMA}.
The target group is equal to
$$\Ext^1_{\Gamma}(\bZ_{\ell},H^{2p-1}(\bar{X},\bZ_{\ell}(p)))$$ 
so we may compare with the extension (\ref{eq:eta}).
When $p=n=\dim(X)$ we obtain
$$\psi^{n}_{\ell}:\CH^n(X)_{hom} \ra 
H^1_{cont}(\Gamma,H^{2n-1}(\bar{X},\bZ_{\ell}(n))).$$
When $k$ is finitely generated over $\bQ$, the Mordell-Weil theorem yields
an injection \cite[9.14]{JannsenBook}
\begin{equation} \label{eq:Jannsen2}
\Alb(X)(k) \otimes \bZ_{\ell} \hookrightarrow 
H^1_{cont}(\Gamma,H^{2n-1}(\bar{X},\bZ_{\ell}(n))).
\end{equation}

Suppose we are given a smooth, projective, and geometrically connected $B$
of dimension $b$. Given
a flat family of codimension-$p$ subschemes
$$\begin{array}{ccc}
\cZ &\hookrightarrow & X\times B \\
\downarrow  & & \\
B & & 
\end{array}
$$
flat pullback followed by pushforward induces
\begin{align*}
H^{2b-1}_{cont}(B,\bZ_{\ell}(b))&\ra H^{2p-1}_{cont}(X,\bZ_{\ell}(p)),\\
\Alb(\bar{B})(\ell) \simeq H^{2b-1}(\bar{B},\bQ_{\ell}/\bZ_{\ell}(b))&\ra 
H^{2p-1}(\bar{X},\bQ_{\ell}/\bZ_{\ell}(p)),
\end{align*}
the latter compatible with Galois actions. 

\subsection{Descent of intermediate Jacobians}
\label{subsect:ACMV}
We recall a result on descending Abel-Jacobi maps:
\begin{theo} \label{theo:ACMV} \cite[Th.~A]{ACMV2}, \cite[Th.~1]{ACMV18}
There exists an abelian variety $J$ over $k$
with the following properties:
\begin{itemize}
\item{there exists an isomorphism
$$\iota: J^p_{cyc}(X_{\bC}) \stackrel{\sim}{\ra} J_{\bC};$$}
\item{given a pointed scheme $(B,0)$ over $k$
that is smooth and geometrically connected
and a family of codimension-$p$ cycles
$$\begin{array}{ccc}
\cZ &\hookrightarrow & X\times B \\
\downarrow  & & \\
B & & 
\end{array}
$$
defined over $k$, there exists a morphism
$$\Phi:B \ra J$$
over $k$ such that 
$$\Phi_{\bC}(b)=\iota \circ \Psi^p_B([\cZ_b]-[\cZ_0]).$$
}
\end{itemize}
Moreover, $J$ is unique up to isomorphism over $k$ and compatible
with field extensions; $\Phi$ is unique and compatible
with field extensions.
\end{theo}
\begin{proof} We sketch the construction of $J$: For each family of codimension-$p$
cycles over a smooth base $B$ defined over $k$
$$\begin{array}{ccc}
\cZ &\hookrightarrow & X\times B \\
\downarrow  & & \\
B & & 
\end{array}
$$
consider the induced 
$$\Alb(\bar{B})(\ell) \ra
H^{2p-1}(\bar{X},\bQ_{\ell}/\bZ_{\ell}(p)).$$
If the family is defined over a finite extension $L/k$, there is an induced
family of cycles over the restriction of scalars 
$$\cZ' \ra \mathbf{R}_{L/k}(B)$$
obtained by summing cycles over conjugate points. The
associated
$$\mathbf{R}_{L/k}(\Alb(\bar{B}))(\ell) \ra 
H^{2p-1}(\bar{X},\bQ_{\ell}/\bZ_{\ell}(p))$$
contains the image of the original homomorphism.
Thus maximal rank images are achievable over the ground field. 
There exists an abelian variety $J'$ over $k$, and the cycle class map induces a
surjection
$$J'_{\bC} \twoheadrightarrow J^p_{cyc}.$$

Consider the induced homomorphisms of torsion subgroups
$$J'_{\bC}[\ell^{\nu}] \ra J^p_{cyc}[\ell^{\nu}] \subset J^p(X)[\ell^{\nu}].$$
The associated Galois data is encoded by
$$\rho_{\ell}: \bar{J}'(\ell) \ra H^{2p-1}(\bar{X},\bQ_{\ell}/\bZ_{\ell}(p)),$$
which is compatible with $\Gamma$-actions.
Descent for homorphisms of abelian varieties -- quotients of a given
abelian variety over $k$ can be read off from 
its $\ell$-adic representations --
yields a unique $J'\twoheadrightarrow J^p_{cyc}$ factoring
$$
\rho_{\ell}: \bar{J}'(\ell) \ra J(\ell) \hookrightarrow 
H^{2p-1}(\bar{X},\bQ_{\ell}/\bZ_{\ell}(p)),
$$
for every $\ell$. Indeed, first find an isogeny
$J''\ra J^p_{cyc}$ and then quotient out by torsion subgroups in the kernel.
The resulting $J$ is defined over $k$ because each $\rho_{\ell}$ is Galois
invariant. \end{proof}

\begin{prop} \label{prop:extension1}
Retain the set-up of Theorem~\ref{theo:ACMV},
with $B$ smooth and geometrically connected over $k$
and $\cZ \subset X\times B$ a family of codimension-$p$
cycles. 
Then for each $d$ there is a morphism over $k$
$$\phi: \Sym^d(B) \times \Sym^d(B)  \ra J$$
such that
$$\phi_{\bC}(\sum_{i=1}^d b_i,\sum_{i=1}^d b'_i) = 
\iota \circ \Psi^p_{\Sym^d(B)\times \Sym^d(B)}(\sum_i [\cZ_{b_i}]-\sum_i[\cZ_{b'_i}]).$$
This morphism is compatible with field extensions.
\end{prop}
\begin{proof} 
First, pass to a resolution $B^{[d]} \ra \Sym^d(B)$;
our family of cycles $\cZ$ induces a family of cycles over 
$B^{[d]}$ by summing over $d$-tuples of points
and taking differences.
Choose a field extension $L/k$ so that $B^{[d]}$ admits an $L$-rational
point. Theorem~\ref{theo:ACMV} gives a morphism
$$\Phi_L:(B^{[d]}\times B^{[d]})_L  \ra J_L;$$
translate in $J_L$ so it takes the diagonal to zero.
Then Galois descent yields a morphism over $k$
$$\phi^{\nu}: B^{[d]} \times B^{[d]} \ra J$$
which induces a morphism on symmetric powers
$$
\phi:\Sym^d(B) \times \Sym^d(B) \ra J,
$$
by the standard Stein factorization argument.
\end{proof}
\begin{coro}
Retain the assumptions of Proposition~\ref{prop:extension1} and
assume there is a $k$-rational point
$\sum_{i=1}^d b_i \in \Sym^d(B)$.
Then there is a morphism over $k$
$$\Phi: \Sym^d(B)  \ra J$$
such that
$$\Phi_{\bC}(\sum_{i=1}^d b_i) = 
\iota \circ \Psi^p_{\Sym^d(B)}(\sum_i [\cZ_{b_i}]-\sum_i[\cZ_{b^{\circ}_i}]).$$
\end{coro}

\begin{prop} \label{prop:extension2}
Let $B$ be seminormal and geometrically connected over $k$
and fix a family of codimension-$p$ cycles
$$\begin{array}{ccc}
\cZ &\hookrightarrow & X\times B \\
\downarrow  & & \\
B & & 
\end{array}
$$
as before. Then there is a morphism over $k$
$$\phi:  B \times B   \ra J$$
such that
$$\psi_{\bC}(b,b') = 
\iota \circ \Psi^p_{B\times B}([\cZ_b]-[\cZ_{b'}]).$$
This is compatible with field extensions.
\end{prop}
\begin{proof}
Pick a normal crossings
resolution $\widetilde{B} \ra B$
and field extension $L/k$
over which each stratum is defined and admits a rational point.
We construct 
$$(\widetilde{B} \times \widetilde{B})_L \ra J_L$$
as in the proof of Proposition~\ref{prop:extension1}
with the diagonals mapped to zero. 
Furthermore, if $b_1,b_2 \in \widetilde{B}(L)$ are identified
in $B(L)$ then we map $(b_1,b')$ and $(b_2,b')$ to the same point in $J$.
As $B$ is geometrically connected, the resulting 
$$\widetilde{\phi}_L: (\widetilde{B}\times \widetilde{B})_L \ra J_L$$
is uniquely determined -- without connectivity
the morphism on the `off-diagonal' components
would only be determined up to translation.
This gluing data satisfies the requisite compatibilities
-- we may verify this over $\bC$ where
Proposition~\ref{prop:Chowmap} applies. Thus we have
$${\phi}_L: (B\times B)_L \ra J_L.$$
The requisite gluing relations are induced by morphisms defined
over $k$ so $\phi_L$ descends to the desired morphism over $k$.
\end{proof}

\subsection{Application of the Albanese to cycle maps}
Fix 
$$
\cC \subset \Chow^p_d(X),
$$
a connected component of the Chow variety that
is geometrically connected. 
Let $P_{\cC}$ denote the principal homogeneous space over 
$\Alb(\cC)$ and $i_{\cC}:\cC \ra P_{\cC}$ the morphism 
constructed in Section~\ref{sect:Alb}. 

\begin{theo} \label{theo:key}
There is a 
homomorphism of abelian varieties over $k$
$$\varphi: \Alb(\cC) \ra J,$$
where $J$ is the model defined in Theorem~\ref{theo:ACMV},
with the following property:
Consider the principal homogeneous space
$$
\begin{array}{lcr}
J \times P & \ra & P \\
(j,p) & \mapsto & j\cdot p,
\end{array}
$$
where $[P]=\varphi([P_{\cC}])=P$, and the morphism
$$\Phi:\cC \ra P$$
induced from $i_{\cC}$. For each $c_1,c_2 \in \cC$
and corresponding cycles $\cZ_{c_1}$ and $\cZ_{c_2}$,
we have
$$\Phi_{\bC}(c_2)=\iota(\Psi^p([\cZ_{c_2}]-[\cZ_{c_1}]))\cdot \Phi_{\bC}(c_1).$$
\end{theo}
\begin{proof}
Corollary~\ref{coro:symAlb}
-- the Albanese is dominated by large symmetric powers -- shows that it
suffices to construct compatible morphisms
over symmetric powers of $\cC$.
Propositions~\ref{prop:extension1} and \ref{prop:extension2} explain the
passage to symmetric powers and to singular parameter spaces.
Thus we obtain the homomorphism of abelian varieties over $k$
$$\varphi: \Alb(\cC) \ra J$$
such that each zero cycle 
$\sum_i n_i c_i$ of degree zero goes to the corresponding
cycle class $\iota(\Psi^p(\sum_i n_i\cZ_{c_i}))$ in $J$.
It follows immediately that $\Phi$ admits the desired interpretation
as a cycle class map to a principal homogeneous space over $J$.
\end{proof}

\subsection{Compatibility under addition}
\label{subsect:add}
Let $\cC$ and $\cC'$ denote geometrically connected components
of $\Chow^p$ defined over $k$, so that 
$$
\cC \times_{\Spec(k)} \cC'
$$ 
is geometrically connected as well. It is also
seminormal \cite[5.9]{GrTr}. Let $\cC''$ denote the geometrically
connected component of $\Chow^p$ obtained via addition
$$\begin{array}{rcl}
\alpha: \cC \times_{\Spec(k)} \cC' & \ra & \cC''\\
 (Z,Z') & \mapsto & Z+Z'.
\end{array}
$$
We refer the reader to \cite[I.3.21]{KollarBook} for a discussion
of the representability properties of $\Chow^p$ underlying these
morphisms.

\begin{prop}
\label{prop:add}
Retain the notation of Theorem~\ref{theo:key}.
Suppose that $P$, $P'$, and $P''$ are the principal homogeneous 
spaces over $J$ associated with $\cC$, $\cC'$, and $\cC''$.
Then we have
$$[P]+[P']=[P''].$$
\end{prop}
\begin{proof}
The additivity is clear for the fiber product of $\cC$ and
$\cC'$. The morphism $\alpha$ induces a morphism of the
corresponding principal homogeneous spaces for $J$. It is 
evidently an isomorphism over extensions over
which $\cC$ and $\cC'$ admit rational points.
Thus it must also be an isomorphism over its field of definition.
\end{proof}

\section{Construction of invariants}
\label{sect:construct}
We continue to use the notation of Section~\ref{sect:passage}.

\subsection{Galois actions on cycle groups}
\label{subsect:goal}
\begin{prop}
Fix a Galois extension $k\subset L \subset \bC$ and $\sigma \in \Gal(L/k)$.
\begin{enumerate}
\item 
If  $Z_1$ and $Z_2$ are defined and algebraically equivalent over $L$.
Then ${ }^{\sigma}Z_1$ and ${ }^{\sigma}Z_2$ are as well.
\item
If $Z_1$ and $Z_2$ are defined over $k$ and are algebraically equivalent over $L$
then there exists an $N$ dividing $[L:k]$ such that $NZ_1$ and $NZ_2$ 
are algebraically equivalent over $k$.
\item
If $Z_1$ and $Z_2$ are defined over $L$, are algebraically equivalent over some extension of $L$,
and Abel-Jacobi equivalent to zero then 
${ }^{\sigma}Z_1$ and ${ }^{\sigma}Z_2$ are as well.
\end{enumerate}
\end{prop}
\begin{proof}
The first assertion is trivial. The second is standard: Suppose $C$ is a
smooth connected curve over $L$ with rational points $c_1,c_2 \in C(L)$ admitting
a family of cycles $\cZ \ra B$ with $\cZ_{c_1}=Z_1$ and $\cZ_{c_2}=Z_2$.
Then the restriction of scalars $\mathbf{R}_{L/k}(C)$ is defined over $k$ and
admits a family of cycles
$$\cZ' \ra \mathbf{R}_{L/k}(C)$$
obtained by summing over the conjugates. The fibers over $c_1$ and $c_2$
are $[L:k]Z_1$ and $[L:k]Z_2$ as $Z_1$ and $Z_2$ are Galois invariant. This gives
the desired algebraic equivalent. 

For the third statement, observe that Abel-Jacobi equivalence is
stable under field extensions. Thus we may pass to an extension $L'$ over
which our cycles 
are algebraically equivalent via $C$ as above. We have a morphism $C \ra J_{L'}$
such that $c_1$ and $c_2$ map to the same $L$-rational point of $J$. Then the
same holds after conjugating the points and the morphism.
\end{proof}

The groups $C^p(\bar{X})$, $B^p(\bar{X})$, and $\Griff^p(\bar{X})$
all admit actions of $\Gamma$ compatible with the homomorphisms
and inductive
structures we introduced previously. The situation for $G^p(\bar{X})$ is less
straightforward:
\begin{ques}  \label{ques:AJ}
Is Abel-Jacobi triviality an algebraic notion? Let $X$ be a smooth projective variety
and $Z$ a codimension-$p$ cycle on $X$ homologous to zero,
both defined over a field $k$.
\begin{itemize}
\item{ 
Given embeddings $i_1,i_2:k \hookrightarrow \bC$, if $i_1(Z)$ is Abel-Jacobi
equivalent to zero does it follow that $i_2(Z)$ is as well?
cf.\cite[Conj.~2]{ACMV18}}
\item{Suppose that $k$ is finitely generated over $\bQ$ and assume that
$i(Z)$ is Abel-Jacobi equivalent to zero for some embedding $i$. Does it
follow that 
$$\psi^p_{\ell}(Z)=0 \in 
H^1_{cont}(\Gamma,H^{2p-1}(\bar{X},\bZ_{\ell}(p))),
$$ 
for each $\ell$?}
\end{itemize}
The latter statement over number fields $k$ should be compared to the
Bloch-Beilinson conjectures -- see \cite[Conj~9.12]{JannsenBook} 
and the discussion there for context.
\end{ques}

\subsection{The key homomorphism}
Consider codimension-$p$
cycles on $X$ over geometrically connected 
projective schemes over $k$. We discussed in Section~\ref{subsect:add}
how to add two such families. Two families are
{\em equivalent} if they
admit fibers over $\bar{k}$ that are algebraically equivalent.
The resulting group $B^p(X)$ is generated by geometrically-connected connected
components $\cC$ of $\Chow^p$. Indeed, given a family $\cZ\ra B$ over
a geometrically connected base
as indicated, the classifying map from the seminormalization on $B$
$$B^{\nu} \ra \Chow^p$$
maps to such a distinguished such component. 
The fiber map yields an injection
$$B^p(X) \hookrightarrow B^p(\bar{X})^{\Gamma}$$
so this notation is compatible with what was introduced in 
Section~\ref{subsect:Griffiths}.

\begin{exam}
Observe that $B^1(X)=\NS(\bar{X})^{\Gamma}$. Indeed, for sufficiently
ample divisor classes, the corresponding divisors are parametrized
by a Brauer-Severi scheme over a principal homogeneous space for 
the identity component of the Picard scheme. This parameter
space is geometrically integral.
\end{exam}
An element $\zeta \in B^p(X)$ need not be represented by a 
cycle over $k$ -- the base of the family representing $\zeta$ might not admit rational
points.

\begin{theo} \label{theo:tau}
Let $J$ be the abelian variety produced in 
Theorem~\ref{theo:ACMV}.
Then there is a homomorphism
$$\tau:B^p(X) \ra H^1_{\Gamma}(\bar{J})$$
with the following properties:
\begin{itemize}
\item{For each $\zeta \in B^p(X)$, there is an isomorphism
$$\iota_{\zeta}:(P_{\tau(\zeta)})_{\bC} \stackrel{\sim}{\ra} J^p_{cyc}(X_{\bC})+\zeta$$
of $J^p_{cyc}(X_{\bC})$ principal homogeneous spaces.}
\item{
Given a flat family of codimension-$p$ cycles
over a geometrically connected base
$$\begin{array}{ccc}
\cZ &\hookrightarrow & X\times B \\
\downarrow  & & \\
B & & 
\end{array}
$$
there is a morphism
$$\Phi:B \ra P_{\tau([\cZ_b])}$$
over $k$ such that
$$\iota_{[\cZ_b]}\circ \Phi_{\bC} = \Psi^p_{B_{\bC}}.$$
}
\item{These structures are compatible with addition of cycles and field
extensions.}
\end{itemize}
\end{theo}
This is a reformulation of Theorem~\ref{theo:key}. The compatibility
under addition is Proposition~\ref{prop:add}.

\begin{rema} \label{rema:elladic}
The $\ell$-primary parts of this homomorphism are natural from the 
perspecive of the Jannsen's $\ell$-adic Abel-Jacobi map (\ref{eq:Jannsen})
$$
\psi^p_{\ell}:\CH^p(X)_{hom} \ra
H^1_{cont}(\Gamma,H^{2p-1}(\bar{X},\bZ_{\ell}(p))).
$$
Let $r=\dim J$ and consider the homomorphism of Galois representations
$$
H^{2r-1}(\bar{J},\bZ_{\ell}(r)) \ra H^{2p-1}(\bar{X},\bZ_{\ell}(p))$$
associated with the construction of $J$. This is the $\ell$-adic 
analog of the homomorphism arising from the inclusion
$$J_{\bC}=J^p_{cyc}(X_{\bC}) \hookrightarrow J^p(X_{\bC})$$
of complex tori. Fixing a class in $B^p(X)$ allows us to restrict
the 1-cocyle from $H^{2p-1}(\bar{X},\bZ_{\ell}(p))$ to 
$$H^{2r-1}(\bar{J},\bZ_{\ell}(r)).$$
Our construction shows that the cocycle lies in the image of the homomorphism 
cf.~(\ref{eq:Jannsen2}):
$$
T_{\ell}\bar{J} \rightarrow 
H^1_{cont}(\Gamma,H^{2r-1}(\bar{J},\bZ_{\ell}(r))),
$$
where $T_{\ell}$ is the Tate module. The inclusion
$$T_{\ell}\bar{J} \hookrightarrow \bar{J}$$
yields principal homogeneous spaces for $J$ over $k$ with
$\ell$-primary order.
\end{rema}

\begin{rema} \label{rema:reducetoG}
One expects that these invariants should vanish on cycles 
Abel-Jacobi equivalent to zero, i.e., $\tau$ factors through
$G^p(X)$. However, this should follow from a positive resolution
of both parts of Question~\ref{ques:AJ}.
\end{rema}

\subsection{Sample cases}

Suppose that $p=1$ so that
$$\tau:\NS(\bar{X})^{\Gamma} \ra H^1_{\Gamma}(\Pic^0(\bar{X}))$$
is the tautological map assigning a Galois-invariant connected 
component to the associated principal homogeneous space over
the identity component. The image is finite because classes
of divisors over the ground field form a finite-index subgroup of 
the source group.

Suppose that $p=n=\dim(X)$ so that
$$\tau: H^{2n}(X_{\bC},\bZ) \ra H^1_{\Gamma}(\Alb(\bar{X}))$$
assigning to each degree the corresponding principal
homogeneous space over the Albanese (cf.~Corollary~\ref{coro:symAlb}).

The vanishing results in Section~\ref{subsect:vanishing} allow
some sharpened statements:
\begin{itemize}
\item{
Suppose that $p=2$ and write
$$N^2(\bar{X})=\CH^2(\bar{X})/\CH^2(\bar{X})_{hom} \subset \Hg^2(X_{\bC}).$$
If the Chow group of zero cycles on $X$ is supported
on a surface then $\Griff^2(X_{\bC})=0$ and 
$$
B^2(X) \subset N^2(\bar{X})^{\Gamma}
$$
with finite index.}
\item{
If $X$ is a uniruled threefold then the integral
Hodge conjecture holds \cite{VoisinIHC} and
$$N^2(\bar{X})=N^2(X_{\bC})=\Hg^2(X_{\bC}).$$}
\item{
If $X$ is a rationally connected threefold then 
$$
H^4(X_{\bC},\bZ)=\Hg^2(X_{\bC}),
$$ 
and 
$$
B^2(X)\subset H^4(X_{\bC},\bZ)^{\Gamma}
$$ 
with finite index.
The Galois action on $H^4$ reflects the fact that the
cohomology is generated by algebraic cycles. 
The homomorphism $\tau$ has finite image because the classes of
intersections of divisors over the ground field span a finite
index subgroup of $H^4(X_{\bC},\bZ)^{\Gamma}$.}
\item{If $X$ is a prime Fano threefold then the variety of
lines on $X$ is geometrically connected by the
classification \cite{IP}. In this case 
$$
B^2(X)=H^4(X_{\bC},\bZ),
$$
hence 
$$\tau:H^4(X_{\bC},\bZ) \ra H^1(\Gamma,\bar{J}).$$
}
\end{itemize}

\section{Threefolds}
\label{sect:threefolds}

Our main interest in these invariants is in their application
to rationality questions for threefolds that are geometrically
rational. We spell out their birational implications
and explore these in representative examples.

\subsection{A preliminary result}
\begin{prop}
If $X$ is a smooth projective threefold, rational over $k$,
then $B^2(X)=H^4(X_{\bC},\bZ)^{\Gamma}$.
\end{prop}
\begin{proof}
This boils down to two observations, valid for arbitrary fields
$k$:
\begin{itemize}
\item{given a Galois-invariant collection of points 
$S=\{s_1,\ldots,s_r\} \in \bP^3$, there exists a smooth rational
curve in $\bP^3$ containing $S$ and defined over $k$;}
\item{given a smooth projective curve $A\subset \bP^3$ and
$e\in \bZ$,
there exists a geometrically integral family of rational
curves intersecting $A$ in a generic configuration of $e$
points without multiplicity.}
\end{itemize}
The first assertion is a standard interpolation result; the second
follows from the first by working over the function field $L$ of
$\Sym^e(A)$, yielding a rational curve defined over $L$
with the desired incidences.

The general argument is inspired by Example~1.4 and Proposition~4.7
of \cite{KollarLooking}; the latter statement establishes
the birational nature of this interpolation property. 

Consider the birational map $\bP^3 \dashrightarrow X$ and a factorization
$$
\begin{array}{rcccl}
 & & Y & &  \\
& \stackrel{\beta}{\swarrow} &  & \stackrel{\beta'}{\searrow} & \\
\bP^3 & & & & X 
\end{array}
$$
where $\beta$ is a sequence of blow-ups along smooth centers,
defined over $k$. Pushing forward by $\beta'$ gives a split surjection
$$H^4(Y_{\bC},\bZ) \twoheadrightarrow H^4(X_{\bC},\bZ),$$
compatible with Galois actions. Thus 
$$H^4(Y_{\bC},\bZ)^{\Gamma} \twoheadrightarrow H^4(X_{\bC},\bZ)^{\Gamma}$$
and families of cycles in $Y$ over a geometrically connected
base project down to such cycles in $X$.

It suffices to establish the following interpolation result:
Let $S\subset Y$ denote a Galois-invariant collection of smooth
points in the exceptional locus of $\beta$, over a field $L$;
then there exists a smooth rational curve in $Y$ over $L$
meeting the exceptional locus precisely along $S$ with multiplicity one.
Choose formal arcs of smooth curves in $Y$ over $L$ 
transverse to the exceptional locus at $S$. Post-composing by $\beta$
yields formal maps of smooth curves to $\bP^3$. Morphisms 
$\bP^1 \ra \bP^3$ approximating these to sufficiently high order
-- but otherwise disjoint from the center of $\beta$ -- have
proper transforms in $Y$ with the desired intersection property. 
We are free to take the image curves in $\bP^3$ to arbitrarily
large degree, so Lagrange interpolation allows us to produce the
desired curves over $L$.
\end{proof}

\begin{rema} \label{rema:tau}
Thus the invariant $\tau$ -- introduced in Theorem~\ref{theo:tau} --
is defined on $H^4(X_{\bC},\bZ)^{\Gamma}$ when $X$ is rational over $k$.
\end{rema}

\subsection{Rationality criterion}
Let $X$ be a smooth and projective threefold over $k\subset \bC$
such that $X_{\bC}$ is rational. It follows then that $H^3(X_{\bC},\bZ)$
is torsion-free and $J^2(X_{\bC})$ is a principally polarized
abelian variety isomorphic to a (nonempty) product of Jacobians of curves. 
Let $J$ denote its model over $k$ from section~\ref{subsect:ACMV}.

Benoist and Wittenberg have established
\cite[Cor.~2.8]{BenWit} that $J$ is isomorphic to the Jacobian of a 
smooth projective (not necessarily geometrically connected)
curve over $k$ whenever $X$ is rational over $k$.

We state a refinement of the results of \cite[\S~11.5]{HT2quad}
developed in discussion with Wittenberg:
\begin{theo}
Retain the notation introduced above and assume that $X$ is rational
over $k$. Let
$$\tau: H^4(X_{\bC},\bZ)^{\Gamma} \ra H^1_{\Gamma}(\bar{J})$$
denote the invariant constructed in Theorem~\ref{theo:tau}. Then
there exists a smooth projective curve $C$ with positive genus components and
an isomorphism 
$$
i:J\ra J^1(C),
$$
over $k$, along with a Galois
equivariant homomorphism
$$
b:H^2(C_{\bC},\bZ) \ra H^4(X_{\bC},\bZ),
$$
such that the composition
$$H^2(C_{\bC},\bZ)^{\Gamma} \stackrel{b}{\ra} H^4(X_{\bC},\bZ)^{\Gamma} 
\stackrel{\tau}{\ra} H^1_{\Gamma}(\bar{J}) \stackrel{i}{\ra}
H^1_{\Gamma}(\overline{J^1(C)})$$
is the canonical homomorphism assigning each component of $C$ over $k$ to the 
corresponding principal homogeneous space over its Jacobian $J^1(C)=\Alb(C)$.
\end{theo}
On the notation: $C$ is defined over $k$ so that $\Gamma$ acts via
permutation on its geometric components and thus on $H^2(C_{\bC},\bZ)$. 
\begin{proof}
Our approach follows \cite{ManinMotive} in spirit.

Consider the birational map $\bP^3 \dashrightarrow X$ and a factorization
$$
\begin{array}{rcccl}
 & & Z & &  \\
& \stackrel{\beta}{\swarrow} &  & \stackrel{\beta'}{\searrow} & \\
\bP^3 & & & & X 
\end{array}
$$
where $\beta$ is a sequence of blow-ups along smooth centers,
defined over $k$. The blow-up formula \cite[{\S}6.7]{Fulton} tells us
that
$$\CH^2(\bar{Z}) = \bZ \oplus P \oplus (\oplus_{i=1}^N \CH^1(\bar{A_i})) $$
where $P$ is a permutation module associated with the points blown
up and the $A_i$ are the smooth irreducible curves blown up. 

Consider the $A_i$ that `survive' in $X$, i.e., positive genus
curves whose Jacobians
appear as principally polarized
factors in the intermediate Jacobian of $X$. A positive genus
curve $A_i$ that does not survive
is explained by $\beta'$ blowing down a curve $A'$ with $J^1(A')\simeq J^1(A)$.  
Let $C$ denote the union of the surviving $A_i$, which is 
$\Gamma$-invariant, as the Galois action respects the decomposition of
the intermediate Jacobians into simple factors.  
It follows that $J\simeq J^1(C)$. 
Let $b$ be obtained by assigning to each surviving 
$A_i$ the total transforms in $X$ of
the exceptional fibers over $A_i$ at the point where it
is blown up in $\beta$. Thus we obtain that 
$$\CH^2(\bar{X}) = P' \oplus \CH^1(\bar{C})$$
where $P'$ is a permutation module reflecting punctual and genus zero
centers of $\beta$. 

Suppose we are given a family of curves 
$$\begin{array}{ccc}
\cZ &\hookrightarrow & X\times B \\
\downarrow  & & \\
B & &
\end{array}
$$
over a geometrically connected base $B$. Passing to residual curves
in complete intersections on $X$ if necessary, we may assume that the 
generic fibers intersect the total transforms of
the exceptional divisors associated with the components
$C_1,\ldots,C_r\subset C$ properly. Intersecting, we obtain a morphism
$$\sigma: B \ra \Pic^{e_1}(C_1) \times \cdots \times \Pic^{e_r}(C_r),$$
where $e_1,\ldots,e_r$ depend only on the homology class
of $[\cZ_b]$. Over $\bC$, $\sigma$ coincides with the
cycle map $\psi^2$ from $B$ to
the appropriate component of $E^2(X_{\bC})$, a principal
homogeneous space over $J^2(X_{\bC})$. 

This geometric construction shows that $\tau$ fits into the stipulated
factorization, i.e., the principal homogeneous space for
$J$ receiving the cycles parametrized by $B$ is necessarily of the
form
$$\Pic^{e_1}(C_1) \times \cdots \times \Pic^{e_r}(C_r),$$
where the $e_i$ depend on the homology classes of the corresponding
curves.
\end{proof}

\begin{rema}
\begin{enumerate}
\item{This result is most useful when $C$ is determined uniquely by $X$.
This follows from the Torelli Theorem \cite{SerreLauter} provided
the geometric components of $C$ all have genus at least two. Indeed,
over nonclosed fields there may be numerous genus one curves with a 
given elliptic curve as their Jacobian.}
\item{It would be interesting to have explicit examples of 
rational threefolds $X$ admitting a diagram
$$\begin{array}{rcccl}
	&	& X &	& \\
        &\stackrel{\beta}{\swarrow} & & \stackrel{\beta'}{\searrow} & \\
  \bP^3 &   			    & &				    & \bP^3
\end{array}
$$
where $\beta$ and $\beta'$ are blowups along smooth centers over $k$,
satisfying the following:
\begin{itemize}
\item{the only positive genus centers of $\beta$ and $\beta'$ are
irreducible genus one curves $E$ and $E'$;}
\item{$E$ and $E'$ are not isomorphic over $k$;}
\item{$J^1(E) \simeq J^1(E')$ in such a way that the subgroups 
$\bZ[E]$ and $\bZ[E']$ in
the Weil-Ch\^atelet group coincide
cf.~\cite{AKW}.}
\end{itemize}
The order of $[E]$ in the Weil-Ch\^atelet group must be at least five.
Examples in this vein, with centers K3 surfaces instead of elliptic curves, 
exist for complex fourfolds \cite{HasLai}.
}
\item{If $C$ is a geometrically irreducible curve
of genus $g>1$ then $\Pic^e(C)$ has order dividing $2g-2$,
as the canonical divisor gives $\Pic^{2g-2}(C) \simeq J^1(C)$.}
\end{enumerate}
\end{rema}

\subsection{Complete intersections of two quadrics}
Our invariant gives the following extension of Theorem~36
of \cite{HT2quad}:
\begin{theo} \label{theo:twoquad}
Let $X \subset \bP^5$ be a smooth complete intersection of
two quadrics over a field $k\subset \bC$. Then $X$ is rational
over $k$ if and only if $X$ admits a line defined over $k$.
\end{theo}
We refer the reader to \cite[{\S}11]{HT2quad} for specific situations
where there are no lines, e.g., the 
isotopy classes over $\bR$ not admitting lines.
\begin{proof}
The reverse implication is classical so we focus on proving that
every rational $X$ admits a line.

The behavior of our invariant $\tau$ was analyzed by X.~Wang
in \cite{wang,BGW}:
\begin{itemize}
\item{$J\simeq J^1(C)$, where $C$ is the genus two curve associated
with the pencil of quadrics cutting out $X$;}
\item{the variety of lines $F_1(X)$ is a principal homogeneous
space over $J^1(C)$ satisfying 
$$2[F_1(X)]=[\Pic^1(C)].$$}
\end{itemize}
Assuming $X$ is rational, there exists a genus two curve
$C'$ blown up in $\bP^3 \dashrightarrow X$ such that 
$F_1(X)\simeq \Pic^e(C')$ for some degree $e$. However,
the Torelli Theorem \cite{SerreLauter} implies that $C\simeq C'$,
whence 
$$2[\Pic^e(C)]=[\Pic^1(C)].$$
It follows that $\Pic^1(C)$ and $F_1(X)$ are trivial
as principal homogeneous spaces over $J^1(C)$, i.e.,
$$F_1(X)(k) \neq \emptyset.$$
Thus we obtain a line over $k$.
\end{proof}

The geometry of rational curves on $X$ is a good testing ground 
for the constructions underlying the formulation of $\tau$:
\begin{itemize}
\item{$\Chow^2_1(X)$ coincides with $F_1(X)$.}
\item{$\Chow^2_2(X)$ admits two components, $\Sym^2(F_1(X))$ and
the variety of conics which is an \'etale $\bP^3$-bundle over $C$
\cite[{\S}~2]{HT2quad} -- these meet along a Kummer surface bundle
over $C$ with fibers realized as $16$-nodal quartic surfaces.
Both map naturally to the same principal homogeneous space over
$J^1(C)$, which may be interpretted as both $2[F_1(X)]$ and 
$\Pic^1(C)$.}
\item{$\Chow^2_3(X)$ admits three components
\begin{enumerate}
\item{$\Sym^3(F_1(X))$;}
\item{the product of $F_1(X)$ and the variety of conics;}
\item{the variety of rational cubic curves, which carries
the structure of a $\Gr(2,4)$-bundle over $F_1(X)$ \cite[{\S}4.2]{HT2quad}.}
\end{enumerate}
}
\item{$\Chow^2_4(X)$ admits a number of components -- in addition
to those parametrizing reducible curves we have the rational normal
quartic curves in $X$ and the codimension-two linear sections of $X$,
both of dimension eight. These together map naturally to the trivial
principal homogeneous space over $J^1(C)$ although only the latter
obviously admits a rational point.}
\end{itemize}

\begin{rema}
Kuznetsov proposes 
\cite[{\S}2.4]{KuzSurvey}
invariants of Fano threefolds $X$ with $J^2(X)\simeq J^1(C)$,
relating the derived category $\mathsf D^b(X)$ to derived categories
of twisted sheaves on $C$. The Brauer group of $C$ is related to
principal homogeneous spaces over $J^1(C)$. It would be interesting
to compare this approach with our invariant.
\end{rema}

\bibliographystyle{alpha}
\bibliography{Cycle}

\begin{thebibliography}{ACMV18b}

\bibitem[ACMV18a]{ACMV18}
Jeffrey~D. Achter, Sebastian Casalaina-Martin, and Charles Vial.
\newblock Algebraically motivated normal functions are algebraic, 2018.
\newblock \texttt{arXiv:1810.07404}.

\bibitem[ACMV18b]{ACMV2}
Jeffrey~D. Achter, Sebastian Casalaina-Martin, and Charles Vial.
\newblock Distinguished models of intermediate {J}acobians.
\newblock {\em J. Inst. Math. Jussieu}, 2018.

\bibitem[AKW17]{AKW}
Benjamin Antieau, Daniel Krashen, and Matthew Ward.
\newblock Derived categories of torsors for abelian schemes.
\newblock {\em Adv. Math.}, 306:1--23, 2017.

\bibitem[BGW17]{BGW}
Manjul Bhargava, Benedict~H. Gross, and Xiaoheng Wang.
\newblock A positive proportion of locally soluble hyperelliptic curves over
  {$\Bbb Q$} have no point over any odd degree extension.
\newblock {\em J. Amer. Math. Soc.}, 30(2):451--493, 2017.
\newblock With an appendix by Tim Dokchitser and Vladimir Dokchitser.

\bibitem[Blo79]{Bloch}
S.~Bloch.
\newblock Torsion algebraic cycles and a theorem of {R}oitman.
\newblock {\em Compositio Math.}, 39(1):107--127, 1979.

\bibitem[BS83]{BlSr}
S.~Bloch and V.~Srinivas.
\newblock Remarks on correspondences and algebraic cycles.
\newblock {\em Amer. J. Math.}, 105(5):1235--1253, 1983.

\bibitem[BW19]{BenWit}
Olivier Benoist and Olivier Wittenberg.
\newblock The {C}lemens-{G}riffiths method over non-closed fields, 2019.
\newblock \texttt{arXiv:1903.08015}.

\bibitem[Car87]{Carlson}
James~A. Carlson.
\newblock The geometry of the extension class of a mixed {H}odge structure.
\newblock In {\em Algebraic geometry, {B}owdoin, 1985 ({B}runswick, {M}aine,
  1985)}, volume~46 of {\em Proc. Sympos. Pure Math.}, pages 199--222. Amer.
  Math. Soc., Providence, RI, 1987.

\bibitem[EV88]{EV}
H\'{e}l\`ene Esnault and Eckart Viehweg.
\newblock Deligne-{B}e\u{\i}linson cohomology.
\newblock In {\em Be\u{\i}linson's conjectures on special values of
  {$L$}-functions}, volume~4 of {\em Perspect. Math.}, pages 43--91. Academic
  Press, Boston, MA, 1988.

\bibitem[Ful98]{Fulton}
William Fulton.
\newblock {\em Intersection theory}, volume~2 of {\em Ergebnisse der Mathematik
  und ihrer Grenzgebiete. 3. Folge. A Series of Modern Surveys in Mathematics}.
\newblock Springer-Verlag, Berlin, second edition, 1998.

\bibitem[Gri69]{GriffithsAnnals}
Phillip~A. Griffiths.
\newblock On the periods of certain rational integrals. {I}, {II}.
\newblock {\em Ann. of Math. (2) 90 (1969), 460-495; ibid. (2)}, 90:496--541,
  1969.

\bibitem[Gri70]{GriffithsIHES}
Phillip~A. Griffiths.
\newblock Periods of integrals on algebraic manifolds. {III}. {S}ome global
  differential-geometric properties of the period mapping.
\newblock {\em Inst. Hautes \'{E}tudes Sci. Publ. Math.}, (38):125--180, 1970.

\bibitem[GT80]{GrTr}
Silvio Greco and C.~Traverso.
\newblock On seminormal schemes.
\newblock {\em Compositio Mathematica}, 40(3):325--365, 1980.

\bibitem[HL18]{HasLai}
Brendan Hassett and Kuan-Wen Lai.
\newblock Cremona transformations and derived equivalences of {K}3 surfaces.
\newblock {\em Compos. Math.}, 154(7):1508--1533, 2018.

\bibitem[HT19]{HT2quad}
Brendan Hassett and Yuri Tschinkel.
\newblock Rationality of complete intersections of two quadrics, 2019.
\newblock \texttt{arXiv:1903.08979}, with an appendix by {J}.-{L}.
  {C}olliot-{T}h\'el\`ene.

\bibitem[IP99]{IP}
V.~A. Iskovskikh and Yu.~G. Prokhorov.
\newblock Fano varieties.
\newblock In {\em Algebraic geometry, {V}}, volume~47 of {\em Encyclopaedia
  Math. Sci.}, pages 1--247. Springer, Berlin, 1999.

\bibitem[Jan88]{JannsenMA}
Uwe Jannsen.
\newblock Continuous \'{e}tale cohomology.
\newblock {\em Math. Ann.}, 280(2):207--245, 1988.

\bibitem[Jan90]{JannsenBook}
Uwe Jannsen.
\newblock {\em Mixed motives and algebraic {$K$}-theory}, volume 1400 of {\em
  Lecture Notes in Mathematics}.
\newblock Springer-Verlag, Berlin, 1990.
\newblock With appendices by S. Bloch and C. Schoen.

\bibitem[Kol96]{KollarBook}
J\'{a}nos Koll\'{a}r.
\newblock {\em Rational curves on algebraic varieties}, volume~32 of {\em
  Ergebnisse der Mathematik und ihrer Grenzgebiete. 3. Folge}.
\newblock Springer-Verlag, Berlin, 1996.

\bibitem[Kol08]{KollarLooking}
J\'{a}nos Koll\'{a}r.
\newblock Looking for rational curves on cubic hypersurfaces.
\newblock In {\em Higher-dimensional geometry over finite fields}, volume~16 of
  {\em NATO Sci. Peace Secur. Ser. D Inf. Commun. Secur.}, pages 92--122. IOS,
  Amsterdam, 2008.
\newblock Notes by Ulrich Derenthal.

\bibitem[Kuz16]{KuzSurvey}
Alexander Kuznetsov.
\newblock Derived categories view on rationality problems.
\newblock In {\em Rationality problems in algebraic geometry}, volume 2172 of
  {\em Lecture Notes in Math.}, pages 67--104. Springer, Cham, 2016.

\bibitem[Lau01]{SerreLauter}
Kristin Lauter.
\newblock Geometric methods for improving the upper bounds on the number of
  rational points on algebraic curves over finite fields.
\newblock {\em J. Algebraic Geom.}, 10(1):19--36, 2001.
\newblock With an appendix in French by J.-P. Serre.

\bibitem[Man68]{ManinMotive}
Yu.~I. Manin.
\newblock Correspondences, motifs and monoidal transformations.
\newblock {\em Mat. Sb. (N.S.)}, 77 (119):475--507, 1968.

\bibitem[Moc12]{Mochizuki}
Shinichi Mochizuki.
\newblock Topics in absolute anabelian geometry {I}: generalities.
\newblock {\em J. Math. Sci. Univ. Tokyo}, 19(2):139--242, 2012.

\bibitem[Ser60]{SerreMorphism}
Jean-Pierre Serre.
\newblock Morphismes universels et vari\'et\'e d'{A}lbanese.
\newblock In {\em S\'{e}minaire {C}. {C}hevalley, 3i\`eme ann\'{e}e: 1958/59.
  {V}ari\'{e}t\'{e}s de {P}icard}, pages ii+182. \'{E}cole Normale
  Sup\'{e}rieure, Paris, 1960.

\bibitem[TZ14]{TianZong}
Zhiyu Tian and Hong~R. Zong.
\newblock One-cycles on rationally connected varieties.
\newblock {\em Compos. Math.}, 150(3):396--408, 2014.

\bibitem[Voi06]{VoisinIHC}
Claire Voisin.
\newblock On integral {H}odge classes on uniruled or {C}alabi-{Y}au threefolds.
\newblock In {\em Moduli spaces and arithmetic geometry}, volume~45 of {\em
  Adv. Stud. Pure Math.}, pages 43--73. Math. Soc. Japan, Tokyo, 2006.

\bibitem[Wan18]{wang}
Xiaoheng Wang.
\newblock Maximal linear spaces contained in the base loci of pencils of
  quadrics.
\newblock {\em Algebr. Geom.}, 5(3):359--397, 2018.

\end{thebibliography}
\end{document}